\documentclass[12pt,a4paper]{amsart}

\usepackage{amsmath,amsthm,amssymb}
\usepackage{float}
\usepackage{url}
\usepackage{tikz}
\usetikzlibrary{arrows,positioning,decorations,decorations.markings,matrix} 

\theoremstyle{plain}
\newtheorem{theorem}{Theorem}
\newtheorem{lemma}[theorem]{Lemma}

\newtheorem{proposition}[theorem]{Proposition}
\newtheorem{corollary}[theorem]{Corollary}
\newtheorem{definition}[theorem]{Definition}
\theoremstyle{remark}
\newtheorem{remark}[theorem]{Remark}
\numberwithin{theorem}{section}

\newcommand{\F}{\mathbb{F}}

\newcommand{\Q}{\mathbb{Q}}
\newcommand{\Z}{\mathbb{Z}}

\newcommand{\Tr}{\textnormal{Tr}}

\newcommand{\Hom}{\textnormal{Hom}}
\newcommand{\Aut}{\textnormal{Aut}}

\newcommand{\GL}{\textnormal{GL}}

\newcommand{\End}{\textnormal{End}}

\newcommand{\II}{\textsc{II}}

\newcommand{\mykill}[1]{}

\begin{document}

\bibliographystyle{plain}

\title{Automorphisms of modular lattices.}
\author{Gabriele Nebe}
\email{nebe@math.rwth-aachen.de}
\address{Lehrstuhl D f\"ur Mathematik, RWTH Aachen University, 52056 Aachen, Germany}

\begin{abstract}
	The methods to classify extremal unimodular lattices with 
	given automorphisms are extended to the situation of modular 
	lattices. 
	A slightly more general notion than the type from the
	PhD thesis \cite{Juergens} is the det-type. 
	The det-type of an automorphism on $L$ determines the
	one of all partial dual lattices of $L$. This easy observation allows 
	to exclude quite a few det-types of automorphisms left
	open in \cite{Juergens}. Passing to suitable 
	$p$-maximal lattices, 
 extremal $\ell $-modular lattices of composite level 
	$\ell = 14$ and $\ell =15$ of dimension $12$ and the ones 
	of level $\ell =6$ and dimension $16$ are classified. 
\end{abstract}

\maketitle

\section{Introduction}

The study of densest sphere packings in 
Euclidean space is a classical mathematical problem. 
Whereas trivial in dimension 1 and easy in dimension 2 
this is already a very hard problem in dimension 3, known as the Kepler problem.
The (uncountable many) densest packings in Euclidean 3-space 
have been classified only 15 years ago by Thomas Hales \cite{Hales}. 
The sphere-packing problem becomes much easier if one restricts to regular (or lattice) 
packings, where the centers of the spheres form a finitely
generated subgroup of Euclidean space, a so called {\em lattice}. 
The densest lattices are known up to dimension 8 and in dimension 24 
(see \cite{CohnKumar}). The papers \cite{Viazovska} and \cite{ViaCohn} 
prove that the lattice packings in dimension 8 and 24 realize in fact 
the densest sphere packings in their respective dimensions. 
The underlying lattices, the famous $E_8$ lattice and the Leech lattice $\Lambda _{24}$,
are extremal even unimodular lattices. 
Even unimodular lattices are positive definite regular integral quadratic forms. 
They exist only when the dimension is a multiple of $8$. 
A full classification is known in dimension 8, 16 and 24, see Section \ref{known}
below. 
The theory of modular forms allows to show that the minimum of an
even unimodular lattice of dimension $n$ is bounded from above by 
$2+2\lfloor \frac{n}{24} \rfloor $, lattices achieving equality are 
called extremal. As the minimum of a unimodular lattice
determines its density, the extremal lattices are the densest even unimodular lattices 
in their dimension.
In the jump dimensions, the multiples of 24, one knows only 6 extremal lattices,
the Leech lattice $\Lambda _{24}$, the unique 24-dimensional unimodular lattice 
of minimum $4$, four lattices, $P_{48q}$, $P_{48p}$ \cite{SPLAG}, $P_{48n}$
\cite{cycloquat}, $P_{48m}$ \cite{Nebe48}, 
in dimension 48 and one lattice, $\Gamma _{72}$ \cite{Nebe72}, of dimension 72. 
In \cite{Nebeautext} the author started a long term project to classify extremal lattices
with given symmetry which also led to the discovery of $P_{48m}$. 
The thesis \cite{Juergens} applied the 
techniques from \cite{Nebeautext} to the more general situation of 
extremal $\ell $-modular lattices (see Definition \ref{modular}).

The present paper aims to give a few easy more general techniques 
for the use of automorphisms to classify such extremal lattices. 
We try to be very brief and not to overload the paper with definitions.
The interested reader is referred to the textbooks 
\cite{Martinet} (for more geometric properties of lattices), 
\cite{Ebeling} (for the relations between lattices and modular forms), 
\cite{Kneser} and \cite{OMeara} (for the arithmetic theory of quadratic forms)
and also the famous collection \cite{SPLAG}. 
The most important notions are given 
in Section \ref{lattices} which also lists the current state
of knowledge on extremal lattices in Section \ref{known}. 
One highlight of the paper is Section \ref{composite}.
Passing to the maximal lattice at a 
suitable prime divisor of $\ell $ allows to classify certain extremal 
$\ell $-modular lattices of composite level $\ell $ in a few minutes computer 
calculations, thus providing new complete classifications to Section \ref{known}. 

The main notion to deal with automorphisms is the one of the {\em type} 
of an automorphism $\sigma $ of prime order $p$ of the lattice $L$ introduced in 
\cite{Nebeautext} (see Section \ref{type}).
It is independent of the quadratic form on $L$ and determines the 
$\Z_p[\sigma ]$-module structure of $\Z_p\otimes _{\Z} L$. 

A finer information than the type is the {\em det-type}, introduced 
in \cite{Juergens}, where it is called 
type.
The det-type of a lattice determines the det-type of all its partial 
dual lattices (Theorem \ref{main1}). For extremal strongly $\ell $-modular lattices 
all these partial dual lattices are again extremal. This observation
allows to exclude quite a few possibilities for det-types of 
automorphisms that are left open in \cite{Juergens}, see Section \ref{dim36} for 
examples.

\section{Lattices}\label{lattices}

Throughout this paper let $(V,Q) $ be a positive definite 
rational quadratic space of dimension $n$. 
So $Q:V\to \Q $ is a rational quadratic form such that 
$Q(v) \geq 0 $ for all $v\in V$ with equality if and only if $v=0$.
The associated bilinear form will be denoted by $(,)$, so
$$(x,y) = Q(x+y) - Q(x) - Q(y) \mbox{ for all } x,y\in V .$$
A {\em lattice} $L$ in $V$ is a finitely generated 
subgroup of $V$ of full rank. As $\Z $ is a principal ideal domain,  there hence 
exists a basis $B:=(b_1,\ldots , b_n)$ of $V$ such 
$$L= \{ \sum _{i=1}^n a_i b_i \mid a_i \in \Z \mbox{ for all } 1\leq i \leq n\}. $$
The basis $B$ is then called a lattice basis of $L$ and its 
Gram matrix $ ((b_i,b_j)) _{1\leq i , j \leq n}$ a Gram matrix of $L$.
The {\em determinant} of $L$ is the determinant of any of its Gram matrices and 
the {\em  minimum} of $L$ is 
$$\min (L) := \min \{ (\lambda ,\lambda )   \mid 0\neq \lambda \in L \} .$$
The lattice $L$ is called {\em even}, if $Q(\lambda) = \frac{1}{2} (\lambda, \lambda ) \in \Z $
for 
all $\lambda \in L$. 
The {\em dual lattice} of $L$ is 
$$ L^{\#} := \{ v\in V \mid (v,\lambda ) \in \Z \mbox{ for all } \lambda \in L \}. $$ 
The dual basis of any lattice basis of $L$ is a lattice basis of $L^{\#}$.
The {\em level} $\ell $ of an even lattice $L$ is the minimal natural number 
$\ell $ such that the rescaled dual lattice 
$$^{(\ell )} L^{\# } := (L^{\#}, \ell (,) )$$ 
is an even lattice.
In particular even lattices of level 1 are precisely the 
even unimodular lattices, they satisfy $L=L^{\#} $.

The following result is well known (see for instance \cite[Proposition 1.3.4]{Martinet}):

\begin{proposition} \label{dualproj} 
	Let $L$ be a lattice in  $V$,
 $U\leq V$ and $\pi_U \in \End(V) $ the 
	orthogonal projection onto $U$ (with kernel $U^{\perp }$). 
	Put $L_U := L\cap U$.
	Then $L_U$ is a lattice in $U$ with $(L_U)^{\#} = \pi_U(L^{\#}) $.
\end{proposition}


\subsection{Genera of lattices and the mass formula}

Two lattices $L,M$ in $(V,Q)$ are called {\em isometric},
if there is an orthogonal mapping 
$$\sigma \in O(V,Q) := \{ \sigma \in \GL(V) \mid Q(\sigma(x)) = Q(x) \mbox{ for all } x\in V \} $$ with 
$\sigma (L) = M$. The stabilizer of $L$ in $O(V,Q) $ is called the
{\em automorphism group} $\Aut(L)$. This is always a finite unimodular group. 

The lattices $L$ and $M$ in $V$ are said to be {\em in the same genus}, if 
the $p$-adic completions $\Z_p\otimes _{\Z} L$ and $\Z_p \otimes _{\Z} M$ 
are isometric for all primes $p$. 
A genus of lattices contains only finitely many isometry classes,
if $L_1, \ldots , L_h$ represent these classes then 
$$\sum _{i=1}^h |\Aut(L_i)|^{-1} $$
is called the {\em mass of the genus}. This rational number can be calculated 
directly from the local invariants of the genus, without knowing 
the $L_i$. As $|\Aut(L_i) | \geq 2$ for all $i$, 
the class number, $h$, is always at least twice the mass of the genus. 
For more details we refer to \cite[Section 102]{OMeara}
or \cite[Chapter VII-X]{Kneser}. 
The latter reference, \cite[Section 28]{Kneser}, also describes an 
algorithm, the Kneser-neighbor-algorithm,  that is used to enumerate 
all isometry classes of lattices in a genus. 
This algorithm is available in {\sc Magma} \cite{Magma}.

\subsection{Extremal strongly modular lattices} \label{modular}
Let $\ell $ be a square-free integer and let $L$ be an even lattice of level $\ell $.
For a divisor $d$ of $\ell $, the 
 the partial dual lattice, $L^{\#,d}$  is defined as 
$L^{\#,d} := L^{\#} \cap \frac{1}{d} L$. 

The lattice $L$  is called {\em strongly $\ell $-modular},
if $L$ is isometric to all its rescaled partial dual lattices, 
i.e. for all $d\mid \ell $ we have $L\cong \ ^{(d)} L^{\#,d }$. 

The notion of strongly $\ell $-modular lattices generalizes the one of 
$p$-modular lattices ($p$ a prime) from \cite{Quebbemann} 
and was introduced in 
\cite{Quebbemann2}.  In this paper Quebbemann shows the following 
generalization of Siegel's \cite{Siegel} result for unimodular lattices.

\begin{theorem} \cite{Quebbemann2}
Let $\ell $ be a square-free integer such that the 
divisor sum $\sigma _1(\ell ):= \sum _{d\mid \ell } d $ 
divides $24$ and let $L$ 
be a strongly $\ell$-modular lattice of dimension $n$.
Then
        $$\min (L) \leq 2+ 2 \lfloor \tfrac{n\sigma_1(\ell)}{24 \sigma_0(\ell)} \rfloor $$
where $\sigma_0(\ell ) := \sum_{d\mid \ell } 1$ is the number of
divisors of $\ell $.
        Strongly $\ell$-modular lattices achieving equality are called \emph{extremal}. 
\end{theorem}

As $\sigma _1(\ell ) $ divides 24, the number 
$J(\ell):= \frac{24 \sigma_0(\ell )}{\sigma _1(\ell )}$  is an integer. 
The extremal $\ell $-modular lattices where the dimension $n$ is a multiple 
of $J(\ell )$ (these dimensions are also called the \emph{jump dimensions})
are of particular interest, as their minimum is strictly bigger than the
one of smaller dimensional $\ell $-modular lattices. 

For a survey of the relation between lattices, modular forms and spherical 
designs the reader might want to consider my two articles 
\cite{NebeJahresbericht} and \cite{NebeVenkovsurv}.

\subsection{The known extremal lattices} \label{known}

As mentioned in the introduction, even unimodular lattices only exist in dimensions
a multiple of 8. Up to dimension 24 all even unimodular lattices are classified, 
for higher dimensions, the mass formula (\cite{SiegelI}, \cite{SiegelIII}, see also 
\cite[Chapter 16]{SPLAG}) 
gives a lower bound for the number of such lattices. 
Oliver King \cite{King} refined this mass formula to count lattices of minimum $>2$, 
which improves these bounds and also provides lower bounds on the number 
of extremal lattices in dimension 32. 
For dimension 40-80 we applied the Minkowski-Siegel mass formula to obtain the rough lower bounds in the table below. 
In the jump dimensions, the multiples of $J(1)=24$, 
we list the number of known extremal lattices. 
In the other dimensions 
the symbol $\exists $
indicates that there are explicit extremal lattices known.

\begin{center}
\begin{tabular}{|c|c|c|@{}c@{}|@{}c@{}|@{}c@{}|@{}c@{}|@{}c@{}|@{}c@{}|@{}c@{}|@{}c@{}|@{}c@{}|}
	\hline
	$n$ & 8 & 16 & 24 & 32 & 40 &  48 & 56 & 64 & 72 & 80   \\
	\hline
	min(L) & 2 & 2 &  4 & 4 & 4   & 6 & 6 & 6 & 8  &  8   \\
	\hline
	$\# $ & 1 & 2 & 24 & $\geq 10^9$  & $ \geq 10^{51} $ & $\geq 10^{121} $  & $\geq 10^{219}$ & $\geq 10^{346} $ & $\geq 10^{506}$ & $\geq 10^{700}$  \\
	\hline
	$\# $ ext & 1 & 2 & 1 & $\geq 10^7$  & $\exists$ & $\geq 4$  & $\exists $ & $\exists $ & $\geq 1$ & $\exists $  \\
	\hline
\end{tabular}
\end{center}

The following table displays the known classification 
of strongly $\ell $-modular lattices for the relevant values of $\ell \geq 2$.
The entries in the table either display the exact number of all 
extremal lattices or a lower bound. 
A ``-'' sign indicates that the extremal modular form has a negative coefficient,
so no such extremal lattice exists. If an entry is empty, then there are no 
strongly $\ell $-modular lattices of the given dimension. 
We also use different colours to indicate the extremal minimum,
\color{blue}{min=2,10}, \color{black}{min=4,12}, {\color{red}{min=6,14}}, {\color{cyan}{min=8,16}} where, of course, the minimum increases downwards in a column.

\begin{center}
	Table of known extremal even strongly modular lattices of 
	levels $\ell = 2,3,5,6,7,11,14,15,23$ and dimension $\geq 48$. \\
\begin{tabular}{|r|c|c|c|c|c|c|c|c|c|}
	\hline
	$J(\ell) $ &  16 & 12 & 8 & 8 & 6 & 4 & 4 & 4 & 2 \\
	$\ell $ & 2 & 3 & 5 & 6 & 7 & 11 & 14 & 15 & 23 \\
\hline
2 & & \color{blue}{1} & & & \color{blue}{1} & \color{blue}{1} & & & 1  \\
4  & \color{blue}{1}  & \color{blue}{1}  & \color{blue}{1}  & \color{blue}{1}  & \color{blue}{1}  & {1}  & {1}  & {1}  &\color{red}{1}   \\
6 & & \color{blue}{1} & & & {1} & {1} & & & \color{cyan}{-}  \\
8  & \color{blue}{1}  & \color{blue}{2}  & {1}  & {1}  & {1}  & \color{red}{1}  & \color{red}{1}  & \color{red}{2}   & \color{blue}{-}  \\
10 & & \color{blue}{3} & & & {4} & \color{red}{2} & &  & {-}  \\
12  & \color{blue}{3}  & {1}  & {4}  & {10}  & \color{red}{0}  & {\color{cyan}{0}}*  & \color{cyan}{1}  & \color{cyan}{3}  & \color{red}{-}  \\
14 & & {1} & & & \color{red}{1} & {\color{cyan}{?}}a & &   & \color{cyan}{-} \\
16 & 1  & 6  & {\color{red}{1}}*  & \color{red}{8}  & \color{red}{$\geq $18}  & \color{blue}{-}  & \color{blue}{?}  & \color{blue}{$\geq $1}   & \color{blue}{-}  \\
18  & & 37 & & & {\color{cyan}{0}}* & \color{blue}{?} & &  & {-} \\
20  & {3}*  & {$\geq $100}  & \color{red}{$\geq $200}a  & \color{red}{$\geq $13}  & \color{cyan}{$\geq $1}a  & {-}  & {?}  & {?}  & \color{red}{-}  \\
22 & & {$\geq 10^3$} & & & {\color{cyan}{?}}a & {?} & &  & \color{cyan}{-} \\
24  & $\geq $8  & \color{red}{$\geq $1}a  & \color{cyan}{$\geq $1}a  & \color{cyan}{$\geq $2}  & {\color{blue}{0}}*  & \color{red}{-}  & \color{red}{?}  & \color{red}{?} & \color{blue}{-}  \\
26 & & \color{red}{$\geq$6} & & & \color{blue}{?} & \color{red}{-} & & & {-}   \\
28  & $\geq ${24}  & \color{red}{$\geq $9}  & \color{cyan}{$\geq $1}  & \color{cyan}{?}  & \color{blue}{?}  & \color{cyan}{-}  & \color{cyan}{?}  & \color{cyan}{?} & \color{red}{-}  \\
30 & & \color{red}{$\geq $2} & & & {-} & \color{cyan}{-} & & & \color{cyan}{-}  \\
32 & \color{red}{$\geq $7}a  & \color{red}{$\geq $33}  & \color{blue}{?}  & \color{blue}{?}  & ?  & \color{blue}{-}  & \color{blue}{-}  & \color{blue}{-}   & \color{blue}{-} \\
34 & & \color{red}{$\geq$100} & & & {?} & \color{blue}{-} & &  & {-}  \\
36  & \color{red}{$\geq $3}  & {\color{cyan}{?}}a  & \color{blue}{?}  & \color{blue}{?}  & \color{red}{-}  & {-}  & {-}  & {-} & \color{red}{-}   \\
38 & & \color{cyan}{?} & & & \color{red}{?} & {-} & &  & \color{cyan}{-}  \\
40 & \color{red}{$\geq $6}  & \color{cyan}{$\geq $1}  & {?}  & {?}  & \color{red}{?}  & \color{red}{-}  & \color{red}{-}  & \color{red}{-}   & \color{blue}{-}  \\
42 & & \color{cyan}{?} & & & \color{cyan}{-} & \color{red}{-} & & & {-}   \\
44  & \color{red}{$\geq $1}  & \color{cyan}{?}  & {?}  & {?}  & \color{cyan}{-}  & \color{cyan}{-}  & \color{cyan}{-}  & \color{cyan}{-} & \color{red}{-}  \\
46 & & \color{cyan}{?} & & & \color{cyan}{?} & \color{cyan}{-} & &  & \color{cyan}{-}  \\
48 & \color{cyan}{$\geq $6}  & \color{blue}{?}  & \color{red}{?}  & \color{red}{?}  & \color{blue}{-}  & \color{blue}{-}  & \color{blue}{-}  & \color{blue}{-}  & \color{blue}{-}  \\
\hline
\end{tabular}
\end{center}

This table is available in the Catalogue of Lattices \cite{LatDB}.

The complete classification results in this table are mostly obtained
by a complete enumeration of the 
strongly $\ell $-modular genus (see \cite{ScharlauHemkemeier}, \cite{SchaSchuPi}).
For certain medium size dimensions, the class number of this genus 
is far too high to enumerate all lattices.
Nevertheless a clever application of  modular forms 
(see \cite{BachocVenkov}) allows to construct all extremal lattices
in the cases marked by $*$. 
An ``a'' indicates that \cite{Juergens} gives restrictions on 
the order of the automorphism group of such lattices. 
  For more 
precise results see below 
or the thesis \cite{Juergens}.

The classification of the extremal strongly 6-modular lattices of dimension 16
and the one of the extremal strongly 14- and 15-modular lattices of dimension 12 is
described in Section \ref{composite} below. 
In particular the classification of extremal strongly $\ell $-modular lattices 
is complete up to dimension $m$ for $(\ell , m)$ as in the following table:

\begin{center}
\begin{tabular}{|r|c|c|c|c|c|c|c|c|c|c|}
	\hline
	$\ell $ & 1  & 2 & 3 & 5 & 6 & 7 & 11 & 14 & 15 & 23  \\
\hline
m & 24 & 20 & 18 & 16 & 16 & 14 & 12 & 12 & 12 & $\infty $ \\ 
\hline
\end{tabular}
\end{center}

In the jump dimensions there is a complete classification of all 
extremal $\ell $-modular lattices for $\ell = 7, 11,$ and $23$.

\section{Maximal lattices} \label{max}

A lattice $L$ in $(V,Q)$ is called {\em maximal} if 
$L$ is even (i.e. $Q(L) \subseteq \Z $) and $L$ is maximal with this property,
i.e. all proper overlattices $M$ of $L$ satisfy 
 $Q(M) \not\subseteq \Z $.
 It is well known (see \cite[Section 82H]{OMeara}) that the set of maximal lattices
 in $(V,Q)$ forms a single genus of lattices. 
 In general this maximal genus has the smallest possible mass among all 
 genera of lattices in $(V,Q)$, and one strategy to find all lattices in 
 a given genus is to construct them as sublattices of the ones in the maximal genus.
 A related strategy is applied in Section \ref{composite} below. 

Corresponding maximality questions in presence of a finite group $G$
are treated in \cite{Morales}. In particular \cite[Lemma (1.2)]{Morales} 
shows that a maximal even $G$-lattice $(M,Q)$ in $(V,Q)$ has a 
semisimple anisotropic discriminant group, so 
$M^{\#}/M $ is a direct sum of simple $\Z G$-modules and 
$\overline{Q}: M^{\#}/M \to \Q/\Z  $ defined by
$\overline{Q}(v+M) := Q(v) + \Z $ does not vanish on any of the non-zero submodules
of $M^{\#}/M$. 
If $G$ is a $p$-group then the simple $\F_pG$-modules are trivial, 
so in this case $(\F_p\otimes _{\Z } M^{\#} / M,\overline{Q}) $ is
an anisotropic quadratic $\F_p$-space allowing to conclude the following theorem.

\begin{theorem}\label{Thmax}
	Let $L$ be an even lattice and 
 $\sigma \in \Aut(L)$ of prime order $p$.
	Then there is a $\sigma $-invariant overlattice $M$ containing $L$ of
	$p$-power index such that $\Z_p \otimes _{\Z} M$ is a 
	maximal lattice in $\Q_p \otimes _{\Q} V$.
\end{theorem}



\section{Strongly modular lattices of composite level}\label{composite}

To classify extremal strongly $\ell$-modular lattices $L$  of composite level 
$\ell = 6,14,15 $ in medium size dimensions one may use the following 
strategy.
If the genus containing the strongly $\ell $-modular lattices is too big to
be enumerated but the classification of $p$-modular lattices $M$ 
is known for some $p$ dividing $\ell $, then one may 
construct $L$ as a sublattice of $\ell/p$-power index in $M$ as illustrated
in this section.

\begin{theorem} \label{dim12ext15} 
	There are exactly three extremal even strongly $15$-modular
	lattices of dimension $12$.
\end{theorem}

\begin{proof}
	Let $L$ be an even 15-elementary lattice of determinant $15^6$
	in dimension 12 such that $\min(L) = 8$ and $\min(L^{\#,3}) = 8/3$. 
	Then there is a 5-elementary lattice $M$ of determinant $5^6$ 
	such that $L\leq M = M^{\#,3} \leq L^{\#,3} $. 
	In particular $\min (M) $ is an even number $\geq 8/3$ so
	$\min(M) \geq 4$. 
	There are 4 such lattices $M$, all are extremal 5-modular lattices. 
	Successively computing 15-elementary sublattices $X$ of index $3^i$, $i=1,\ldots ,3$ of these
	4 lattices $M$, satisfying $\min (X^{\#,3} )\geq 8/3$ we finally obtain three such lattices $L$.  
\end{proof}

\begin{remark}
	More precisely we checked that these three extremal 
	even strongly $15$-modular lattices of dimension $12$
	are the only 
	even $15$-elementary lattices of minimum $8$ and dimension $12$ 
	of determinant $15^6$ such that the $3$-dual 
	has minimum $\geq 8/3$.
\end{remark}

\begin{theorem} 
	There is a unique extremal 
	 strongly $14$-modular lattice in dimension $12$.
\end{theorem} 

\begin{proof}
	Let $L$ be such an extremal lattice. Then $\min (L) =8$ and hence
	$\min (L^{\#,2}) = 4$ and there is a lattice $M$ in the genus of 
	even 7-modular lattices with $L\subseteq M = M^{\#,2} \subseteq L^{\#,2} $,
	in particular $\min (M) \geq 4$. 
	The  genus of $M$ has class number 395 (\cite{ScharlauHemkemeier}, \cite{Magma})
	and contains no lattice with minimum 6 (which would be extremal) 
	and 49 lattices $M$ of minimum 4 such that also the rescaled dual 
	lattice $\ ^{(7)} M^{\#} $ has minimum $4$.
	Successively computing sublattices $X$ of 2-power index in one 
	of these 49 lattices $M$ such that $\min (X^{\#,2}) \geq 4$ 
	one finally reaches a unique strongly 14-modular extremal $L$.
\end{proof}

\begin{theorem} 
	There are exactly $8$ extremal even strongly $6$-modular lattices 
	in dimension $16$.
\end{theorem} 

\begin{proof}
	Let $L$ be such an extremal lattice. Then $\min (L) =6$ and hence
	$\min (L^{\#,2}) = 3$ and there is a lattice $M$ in the genus of 
	even 3-modular lattices with $L\subseteq M = M^{\#,2} \subseteq L^{\#,2} $
	with $\min (M) \geq 4$. These lattices $M$ are hence extremal and 
	there are 6 such lattices. 
	Successively computing sublattices $X$ of 2-power index in one 
	of these 6 lattices $M$ such that $\min (X^{\#,2}) \geq 3$ 
	one finally ends with 8 isometry classes of 
	strongly 6-modular extremal lattices $L$.
\end{proof}

\section{Automorphisms of prime order} \label{autp}

\subsection{The type of an automorphism}\label{type}

Let $L$ be a lattice and $\sigma \in \Aut(L)$ an automorphism of $L$
of prime order $p$.
As $\sigma ^p = 1$ the elements 
$e_1 := \frac{1}{p} (1+\sigma + \ldots + \sigma ^{p-1} )$ and $e_{\zeta }:= 1-e_1$ are orthogonal idempotents in the endomorphism ring of 
the $\Q \langle \sigma \rangle $-module $V:=\Q L$ giving the $\sigma $-invariant
decomposition $$V = V e_1 \oplus V e_{\zeta } =  V_1 \oplus V_{\zeta } 
\mbox{ of dimensions } n_1 := \dim (V_1), n_{\zeta }:= \dim (V_{\zeta }) $$
such that the restriction of $\sigma $ to $V_1$ is the identity and 
the restriction of $\sigma $ to $V_{\zeta }$ has minimal polynomial 
$\Phi _p  := X^{p-1} + X^{p-2} + \ldots + X + 1$. In particular 
$V_{\zeta }$ is a vectorspace over  $\Q[\zeta _p]$, so 
$n_{\zeta } = r_{\zeta }(p-1)$ is divisible by $p-1$.

Put $L_1:=L \cap V_1 := \{ \lambda \in L \mid \lambda \sigma = \lambda \}$ 
the fixed lattice of $\sigma $ in $L$  and 
$$L_{\zeta }:= L \cap V_{\zeta } = \{ \lambda \in L \mid (\lambda , \lambda _1) = 0 \mbox{ for all } \lambda _1 \in L_1 \} ,$$ 
its orthogonal lattice.

As $Le_1 = \pi_{V_1}(L)$ and $Le_{\zeta } = \pi _{V_{\zeta }} (L)$ 
Proposition  \ref{dualproj} yields the following corollary.

\begin{corollary}
	$L^{\#} e_1 = (L_1)^{\#} $ and $L^{\#} e_{\zeta } = (L_{\zeta })^{\#}$.
\end{corollary}

The basis of the definition of the type is given in the following lemma 
(see for instance \cite{Nebeautext}, \cite{Juergens}). 

\begin{lemma} \label{amal}
With the notation above we have 
$$ pLe_1 \perp L e_{\zeta }( 1-\sigma )  \subseteq L_1\perp L_{\zeta } \subseteq L \subseteq Le_1 \perp L e_{\zeta } $$ 
and $L$ is a full subdirect product of $Le_1$ and $L e_{\zeta }$ 
in particular $Le_1/L_1 \cong L e_{\zeta }/L_{\zeta } $ as $\F_p \langle \sigma \rangle $-modules. 
Moreover the integer 
$$s:= \dim _{\F _p} (Le_1/L_1) = \dim _{\F _p} (Le_\zeta /L_\zeta ) $$
satisfies $s\leq \min (n_1,r_{\zeta }) $ 
\end{lemma}

\begin{proof}
Let $R = \Z_p [\sigma] \cong \Z_p C_p \cong \Z_p[X]/(X^p-1) $  
be the group ring of the cyclic group $C_p$ of prime order $p$
over the ring of $p$-adic integers $\Z_p$.
Then the indecomposable $R$-lattices are the free $R$-module $R$,
the trivial $R$-lattice $\Z_p$ and the lattice $\Z_p[\zeta _p]$
in the irreducible faithful $\Q_p$-representation of $C_p$.

By the theorem of Krull-Schmidt, the $R$-lattice 
$\Z_p\otimes _{\Z} L $ is isomorphic to a unique direct sum of 
indecomposable lattices:
$$\Z_p\otimes _{\Z} L  \cong R^{s'} \oplus \Z_p[\zeta _p] ^r \oplus \Z_p^t .$$
The lattice $\Z_p \otimes _{\Z } L_1$ is then a sublattice of 
dimension ${s'}+t$ of $L$ and $$\Z_p \otimes _{\Z } Le_1/\Z_p \otimes _{\Z } L_1
\cong \Z_p/p\Z_p^{s'} \cong Le_1/L_1 \cong \F_p^s ,$$ so $s=s'$ and 
$n_1=s+t \geq s$. Similarly $r_{\zeta } = r+s \geq s$. 
\end{proof}

\begin{definition}
The tuple $p-(r_{\zeta },n_1)-s $ is called the type associated
to $(L,\sigma )$.
\end{definition}

From the proof of Lemma \ref{amal} we obtain the following corollary
(see also \cite[Section 4.1 and 4.2]{Juergens}, in particular 
\cite[Proposition 4.1.8]{Juergens} for part (d)).

\begin{corollary}
	Let $p-(r_{\zeta },n_1)-s $ be the type of $(L,\sigma )$.
	\begin{itemize}
		\item[(a)]  $pLe_1 \subseteq L_1$ and $ Le_{\zeta}(1-\sigma ) \subseteq L_{\zeta }$.
		\item[(b)]
	If $s=n_1$ then $pLe_1= L_1$. 
\item[(c)]
	If $s=r_{\zeta}$ then $Le_{\zeta}(1-\sigma) = L_{\zeta}$. 
\item[(d)] 
	If $p$ does not divide $\det(L)$ then $s\equiv r_{\zeta } \pmod{2}$.
	\end{itemize}
\end{corollary}

\begin{remark}
The type determines the isomorphism class of $L$ as a $\Z_p[\sigma ] $-module. 
In particular the type of $L$ is the same as the one of any of its 
$\sigma $-invariant sublattices of index prime to $p$;
and if $L$ is an even lattice such that $p$ does not divide 
its determinant, then the types of $(L,\sigma )$  and $(L^{\#},\sigma )$
are the same.
\end{remark}

The last sentence of the previous remark 
also holds if $p$ divides the determinant of $L$: 
\begin{lemma} \label{dual} 
	The type of $(L,\sigma )$ equals the type of  $(L^{\#},\sigma )$ and
	also to the type of all its partial dual lattices.
\end{lemma}

\begin{proof}
	The $\sigma $-invariant quadratic form identifies the 
	$L^{\#} $  with $\Hom _{\Z } (L,\Z )$ 
	and also $\Z_p \otimes L^{\#} $ with the dual module 
	$\Hom _{\Z_p} (\Z_p\otimes L , \Z_p) $. 
	As the indecomposable direct summands of 
	$\Z_p\otimes _{\Z} L$ are self-dual, hence  isomorphic 
	to their dual 
	as a $\Z_p[\sigma ]$-module, one sees that $\Z_p\otimes _{\Z} L\cong \Z_p\otimes _{\Z} L^{\#} $ as a $\Z_p[\sigma ]$-module. 
For the partial dual lattice $L^{\#,d}$ it is enough to note that 
$$\Z_p \otimes _{\Z} L^{\#,d} = \left\{ \begin{array}{cc} \Z_p \otimes _{\Z } L^{\#}  & \mbox{ if } p \mid d  \\
\Z_p \otimes _{\Z } L & \mbox{ otherwise.} \end{array} \right. $$
\end{proof}

\subsection{The det-type of an automorphism.} \label{dettype}
Whereas the type of an automorphism determines the 
$\Z_p[\sigma ]$-module structure of $\Z_p\otimes _{\Z} L$, 
the det-type depends on the $\F_q[\sigma]$-module structure 
of the Sylow-$q$-subgroup $L^{\#,q}/L$ 
of the discriminant group $L^{\#}/L$. 
Here we assume that $L$ is an even lattice 
of square-free level $\ell $. 
For all prime divisors $q$ of $\ell $ with $p\neq q$ the Sylow-$q$-subgroup 
of the discriminant groups of $L$ and its sublattice of $p$-power index 
$L_1 \perp L_{\zeta }$ coincide. 
So 
$$L^{\#,q} / L \cong (L^{\#,q})_1/L_1 \perp (L^{\#,q})_{\zeta } / L_{\zeta} \cong (L_1)^{\#,q} / L_1 \perp (L_{\zeta })^{\#,q}/L_{\zeta } $$
and we put  $d_1(q):= \dim _{\F_q} ((L_1)^{\#,q} / L_1  ) $ and 
$d_\zeta (q):= \dim _{\F_q} ((L_\zeta )^{\#,q} / L_\zeta   ) $. 

The irreducible factors of $\Phi_p \in \F_q[X]$ are of degree 
$o_p(q)$, the order of $q$ in $\F_p^{*}$. 
As $(L_\zeta )^{\#,q} / L_\zeta$ is a self-dual $\F_q[\sigma ]$-module 
we find that 
\begin{remark}\label{divdeg}
	The least common multiple of $2$ and $o_p(q)$ divides 
	$d_{\zeta }(q)$. 
\end{remark}

If $p$ divides $\ell $ we also put 
$d_1(p):= \dim _{\F_p} ((L^{\#,p})_1 / L_1  ) $ and
$d_\zeta (p):= \dim _{\F_p} ((L^{\#,p})_{\zeta } / L_\zeta   ) $. 
Then we have for any prime divisor $q$ of $\ell $ that 
$q^{d_1(q)+d_{\zeta}(q)} = |L^{\#,q}/L|$  is the $q$-part
of the determinant of $L$ (for $p=q$ this is 
investigated in more detail in \cite{PreprintNebe}).

\begin{definition}
Let $\sigma $ be an automorphism of order $p$ of an even lattice
$L$ of square-free level $\ell$.
Let $\ell  = q_1\cdots q_r$ be the prime factorization of $\ell $
and put $d_1(q_i):= \dim _{\F_q} ((L^{\#,q_i})_1 / L_1  ) $ and
$d_\zeta (q_i):= \dim _{\F_q} ((L^{\#,q_i})_{\zeta } / L_\zeta   ) $ 
 for $1\leq i \leq r$.
Then the det-type of $(L,\sigma )$ is 
$$[p-(r_{\zeta },n_1)-s,q_1-(d_{\zeta }(q_1),d_{1}(q_1)), \ldots , q_r-(d_{\zeta }(q_r),d_{1}(q_r)) ] .$$
\end{definition}

\begin{theorem} \label{main1}
Let $\sigma $ be an automorphism of order $p$ of an even lattice
$L$ of square-free level $\ell$.
If 
$$[p-(r_{\zeta },n_1)-s,q_1-(d_{\zeta }(q_1),d_{1}(q_1)), \ldots , q_r-(d_{\zeta }(q_r),d_{1 }(q_r)) ] $$
is the det-type of $(L,\sigma )$
then 
the det-type of $(\ ^{(\ell )}L^{\#} ,\sigma )$ is 
$$[p-(r_{\zeta },n_1)-s,q_1-(n_{\zeta }-d_{\zeta }(q_1),n_{1 }-d_{1}(q_1)), \ldots , q_r-(n_{\zeta }-d_{\zeta }(q_r),n_{1 }-d_{1 }(q_r)) ] .$$
\end{theorem} 

\begin{proof}
	By Lemma \ref{dual} the type of $(L,\sigma)$ and the one of $(L^{\#},\sigma )$
	are the same.  To compute the det-type it is enough to deal with one prime
	at a time (formally we could just take the
	tensor product with $\Z_q$ and deal with $q$-adic lattices).
	So  assume that $\ell = q $ is a prime and 
	put $M:=\ ^{(q)}L^{\#} $. 
	Then $M^{\#} = M^{\#,q} = \ ^{(q)} ( \frac{1}{q} L) $ and  
	forgetting the quadratic form 
$(M^{\#})_1 =  \frac{1}{q} L _1 $ 
and $(M^{\#})_{\zeta } =  \frac{1}{q} L _{\zeta } $. 
We have
$$\frac{1}{q} L_1 \supseteq (L^{\#} )_1 \supseteq L_1 \mbox{ and } 
\frac{1}{q} L_{\zeta } \supseteq (L^{\#} )_{\zeta } \supseteq L_{\zeta }  $$
so 
the theorem follows from the fact that 
$$q^{n_1} = |\frac{1}{q}L_1/L_1| = |\frac{1}{q} L_1 / (L^{\#} )_1 |  
| (L^{\#} )_1 / L_1 | $$ and similarly for $L_{\zeta }$.
\end{proof}

\section{Automorphisms of order 2} \label{ord2}

Of course $-1$ is an automorphism of any lattice, the {\em trivial} automorphism of 
order 2, so it is impossible
to exclude automorphisms of order 2. 
However, non-trivial automorphisms of order 2 usually yield quite restrictive conditions, 
e.g. for the extremal even unimodular lattices of dimension 48 there is only 
one possible type of such automorphisms,  $2-(24,24)-24$, 
whereas for automorphisms of order 5 there are three different types occurring 
in the known lattices. 
One reason is the following lemma, which is a direct generalization 
of \cite[Lemma 4.9]{Nebeautext}. 

\begin{lemma}\label{fixlat2}
	Let $M$ be an even lattice such that $M^{\#} / M$ has exponent $2d$ with 
	$d$ odd.
	Then $M$ contains a sublattice $N$ of $2$-power index such that 
	 $\ ^{(1/2)} N =: U$ is an integral lattice and 
	the exponent of $U^{\#}/U$ is $d$. 
	Moreover if $N$  is a proper sublattice of $M$ then 
	$U$ can be chosen to be an odd lattice.
\end{lemma}

\begin{proof}
Since this is a statement about 2-adic lattices, we pass to $M_2:=\Z_2\otimes _{\Z}  M$.
This lattice has a Jordan decomposition 
$M_2 = M_{0} \perp M_1 $ (see for instance \cite[Section 91C]{OMeara}, 
\cite[Chapter 15]{SPLAG}), where
$(M_0,Q)$ is a regular quadratic $\Z_2$-lattice 
of dimension, say, $2m$, and $\ ^{(1/2)} M_1 $ is a regular bilinear $\Z_2$-lattice.
If $m=0$, then $M = N \cong \  ^{(2)} U$ for some integral lattice 
$U$ of odd determinant and we are done.
So assume $m\geq 1$. Then
$M_0$ contains  vectors $v, w$ such that $(v,w) \in \Z_2^*$ and 
we may choose $v$ such that $(v,v)\in 2\Z_2^*$. 
So $\langle v, w \rangle $ is a regular sublattice of $M_0$ and hence
$M_0 = \langle v,w \rangle \perp N_0 $ for some lattice $N_0$.
Then
$$ M_0 \perp M_1 \geq _2 N_0 \perp (\langle v, 2w \rangle \perp M_1)  \cong N_0 \perp N_1$$
with $\dim (N_1) = \dim (M_1) +2$ and $\dim (N_0 ) = 2(m-1)$.
Note that $^{(1/2)} N_1 $ is an odd lattice as $1/2(v,v) \in \Z_2^*$.
Since $(N_0,Q)$ is again regular, we may proceed by induction, until 
$N_0 = \{ 0 \}$. 
Then the  sublattice $N$  of the lattice $M$ is constructed as the
unique lattice $N$ with $\Z_p \otimes _{\Z} N = \Z_p \otimes _{\Z} M$ for 
all primes $p>2$ and $\Z_2 \otimes _{\Z} N = N_1$.
\end{proof}

If $\sigma $ is an automorphism of type $2-(n_{-1},n_{1})-s$ 
of a lattice $L$ of odd determinant,
then $M=L_1(\sigma ) \perp L_{\zeta }(\sigma) $ satisfies the assumption of 
Lemma \ref{fixlat2}, in particular there is a sublattice 
$\ ^{(2)} (U_1\perp U_{\zeta })  \leq M $ with 
$\dim(U_1) = n_1$, $\dim(U_{\zeta }) = n_{-1}$ such that 
$\min (U_1) \geq \frac{1}{2} \min (L)$, 
$\min (U_{\zeta }) \geq \frac{1}{2} \min (L)$
and $U_1^{\#}/U_1 \oplus U_{\zeta }^{\#}/U_{\zeta } \cong L^{\#} / L$ as abelian groups.

\section{Extremal 11-modular lattices of dimension 14}\label{ex11mod}

The paper \cite{NebeVenkov} uses Siegel modular forms of degree 2 
to conclude the non-existence of an even 11-modular lattice of dimension 
12 and minimum 8, so the first dimension where an 11-modular lattice 
of minimum 8 might exist is dimension 14. 
Applying the strategy of Section \ref{max} and Section \ref{ord2}  we obtain:

\begin{theorem}\label{11moddim14}
	Let $L$ be an extremal $11$-modular lattice of dimension $14$.
	Then $\Aut(L) = \{ \pm 1 \} $ is trivial.
\end{theorem}

\begin{proof}
	Satz 4.2.1 in \cite{Juergens} concludes that the only primes 
	that divide the order of $\Aut (L)$ are 2 and 11. 
	\\
	First we assume that there is $\sigma \in \Aut(L)$ of order 11. 
	Then by Theorem \ref{Thmax} there is some maximal lattice 
	$M$ with $\sigma \in \Aut(M)$. As $\det (L) =11^7$ is an odd power of
	$11$ the determinant of $M$ is 11. 
	There is one genus of even lattices of dimension 14 and determinant 11,
	its class number is 8 and there are 4 lattices admitting an
	automorphism of order 11, 2 of which have dual minimum $\geq 8/11$,
	these have dual minimum $12/11$ and the Sylow 11-subgroup of their
	automorphism group $G$ is of order 11. 
For both lattices we take $\sigma $ to be a generator of the Sylow 11-subgroup
of $G$ 
and compute the maximal $\sigma $-invariant sublattices $N$ of $M$.
	It turns out to be easier to dualize the picture: Let $D:=\ ^{(11)}M^{\#} $ denote 
the rescaled dual lattice of $M$ and 
	denote the action of $\sigma $ on  $D^{\#}/D
	\cong \F_{11}^{13}$ by $X \in \F_{11}^{13\times 13}$.
	Then $\ ^{(11)} N^{\#}/D $ is 
	a minimal $\sigma $-invariant isotropic subspace of $D^{\#}/D$. 
	As $\ ^{(11)} N^{\#}/D $  
	is a simple $\F_{11} \langle \sigma \rangle $-module, 
	the action of $\sigma $ on $\ ^{(11)} N^{\#}/D $  is trivial, so 
	$\ ^{(11)} N^{\#}  = \langle D , u \rangle $ where 
	$\langle u+D \rangle \leq D^{\#}/D$ 
	is a one-dimensional subspace in the kernel $K$ of $X-1 \in \F_{11}^{13\times 13}$.
	The normaliser $N_G(\langle \sigma \rangle )$ acts on these one-dimensional subspaces in $K$, admitting only one orbit represented by an 
	isotropic subspace (in both cases). This orbit corresponds 
	to a lattice $\ ^{(11)} N^{\#}  = \langle D , u \rangle $ of 
	minimum 2. As this is smaller than 8, this gives a contradiction. 
	\\
	Is remains to treat automorphisms of 2-power order.
	We first assume that there is $\sigma \in \Aut(L)$ with
	$\sigma ^2 = -1$. Then $\sigma $ acts on $L^{\#}/L \cong \F_{11}^{7}$ 
	with irreducible minimal polynomial $X^2+1 \in \F_{11} [X]$ contradicting
	the fact that 7 is odd (see also Remark \ref{divdeg}). 
	To finish the proof 
	we need to exclude  non-trivial automorphisms of order 2. 
	Let $\sigma $ be such an automorphism of det-type 
	$[2-(n_1,14-n_{-1})-s, 11-(d_1,7-d_{1})]$ (with $s\leq \min(n_1,14-n_1)$, 
	$s\equiv_2 n_1$). 
	Then the det-type of $(\ ^{(11)}L^{\#},\sigma )$ is 
	$[2-(n_1,14-n_{-1})-s, 11-(n_1-d_1, 7-n_{1}+d_{1})]$. 
Applying the strategy of Section \ref{ord2} we hence need to classify quadruples 
of 11-elementary
lattices $(U_{-1},U'_{-1},U_1,U_1')$ of minimum $\geq 4$ and dimensions $(n_1,n_1,14-n_1,14-n_1)$ 
and determinants $(11^{d_1},11^{n_1-d_1},11^{7-d_1},11^{7-n_1+d_1})$.
By symmetry we hence may assume that $n_1\leq 7$ and $d_1\leq n_1/2$. 

By \cite[Chapter 15,Theorem 13]{SPLAG} there are always two genera of positive
definite odd 
11-elementary lattices of given dimension $n$ and determinant $11^d$ 
(except for $n\leq 2$ or $n=d$, where there is just one such genus). 
There is an additional genus of positive definite even lattices if 
$n=2k$ is even and $k\equiv d \pmod{2}$.

Enumerating all these genera we find one possible lattice $U_{-1}$ of 
dimension 4, determinant $11^2$ and minimum 4 and three such lattices $U_{-1}$ 
of dimension 6, determinant $11^3$ and minimum 4. 
For all lattices $U_{-1}$ no proper overlattice of index 2 has minimum 4, so 
$L_{-1}= \ ^{(2)} U_{-1}$ and $s=n_{-1} = \dim (U_{-1}) $ in all four cases and 
$L_{1} $ contains $\ ^{(2)} U_1 $ as a proper sublattice. 
In particular we may choose $U_1$ as an odd lattice of dimension 10 respectively 8
and determinant $11^5$ respectively $11^4$.
The 11-adic lattice
$\Z_{11} \otimes _{\Z} (\ ^{(2)} U_{-1} \perp \ ^{(2)} U_1 ) = \Z_{11} \otimes _{\Z} L $
then determines the genus of $U_1$ completely. 
This allows to enumerate the 
 respective genera for $U_1$ and to  construct all candidates for the lattices 
 $L_{1}$ as overlattices of $\ ^{(2)} U_1$. 
The lattice $L$ is a full 
subdirect product of 
$L_{-1} \perp L_{1} $ and is constructed using the gluing strategy described for instance
in  \cite[Remark 2.5]{KirschmerNebe} and \cite{Nebe48}. 
The computations are done in the 
 master-thesis \cite{Voullie}; no extremal lattice is found.
\end{proof}

%

\section{Extremal 3-modular lattices of dimension 36} \label{dim36}

The existence of an extremal $3$-modular lattice of dimension $36$ is an interesting open problem. On one hand 36 is the only jump dimension, where the 
existence of an extremal lattice of minimum 8 is still open. 
On the other hand such a lattice would yield the densest known sphere packing 
in dimension 36. The third reason comes from a beautiful 
observation by B.Gross 
(see \cite[Introduction]{BachocNebe} for a more detailed explanation):
There is a connection between extremal modular lattices of 
level $7$, $3$, and $1$ based on
 the following chain of division algebras
$$E = \Q [\sqrt{-7} ] \subseteq Q = \left( \frac{-1,-3}{\Q } \right) 
\subseteq {\bf O}  $$
where $Q$  is the definite quaternion algebra over $\Q $ with discriminant 9 and ${\bf O}$ is the  non-associative Cayley octonion algebra over $\Q$. 
All three algebras have a unique conjugacy class of maximal orders
and we obtain unique embeddings 
$$\Z_E \hookrightarrow {\mathcal M} \hookrightarrow {\mathcal O} $$
of these maximal orders. 
For a Hermitian unimodular $\Z_E$-lattice $B$ of dimension $m$,  the 
 trace lattices 
$$\Tr (B) , \Tr (B\otimes _{\Z_E} {\mathcal M}) , 
 \Tr (B\otimes _{\Z_E} {\mathcal O}) $$ 
 are $\Z $-lattices which are $7$-modular of dimension $2m$, 
 $3$-modular of dimension $4m$ respectively unimodular of 
 dimension $8m$. 
 The sequence of dimensions is compatible to the jump dimensions for the
 respective $p$-modular lattices (multiples of $6$, $12$, resp. $24$). 
For $m=3 $ one obtains
the $7$-modular Barnes lattice in dimension $6$, 
the $3$-modular Coxeter-Todd lattice in dimension $12$ and the 
unimodular Leech lattice of dimension $24$.
These three  lattices are extremal of minimum $4$ and they are the densest 
(known) lattices in their respective dimension.
The same construction is applied to a rank $10$ lattice in \cite{BachocNebe}
to obtain extremal lattices of minimum $8$ in 
dimension $20$, $40$ and $80$.
For $m=6$ it is known that no extremal $7$-modular lattice 
of dimension 12 exists \cite{ScharlauHemkemeier} 
 and no extremal 3-modular $\Z $-lattice of dimension 24 
 is the trace lattice of a Hermitian 
 unimodular $\Z[\zeta_3]$-lattice (see \cite{Feit}). 
 Also for $m=9$ 
 the non-existence of an extremal $7$-modular lattice
in dimension  $18$ 
is known \cite{BachocVenkov}, but 
the discovery of a $72$-dimensional extremal
unimodular lattice in \cite{Nebe72} could nevertheless yield the 
hope to find a shorter sequence of 
extremal $3$-modular and unimodular lattices in dimension
$36$ and $72$. 

The dissertation of Michael J\"urgens \cite{Juergens} exhibits the 
possible automorphisms of an extremal $3$-modular lattice in dimension $36$,
whose existence is still open. In particular 
\cite[Section 4.2.3]{Juergens} shows that such an extremal lattice 
has no automorphisms of order $11$, $13$, or any prime $p\geq 23$ and specifies a unique possible 
det-type for automorphisms of order $17$ and $19$ .  
The paper 
\cite{KirschmerNebe} constructs binary Hermitian lattices over $\Z[\zeta_p]$ for
$p=17 $ and $p=19$ to conclude that such automorphisms do not exist.
So we know that the only primes that might divide the order of the
automorphism group of an extremal 3-modular lattice of dimension 36 are 
$\leq 7$.

\begin{proposition} \label{aut7} 
Let $L$ be an extremal $3$-modular lattice of dimension $36$ and 
let  $\sigma \in \Aut(L)$ be an automorphism of order $7$. 
Then $\mu _{\sigma} = \Phi _7$, i.e. $L$ is a lattice of dimension 
6 over $\Z[\zeta _7]$. 
\end{proposition}

\begin{proof}
By \cite{Juergens} the only possible det-types of $(L,\sigma )$ are 
	$$\begin{array}{l} \mbox{ [7-(4,12)-4,3-(12,6)], [7-(5,6)-3,3-(12,6)], } \\ \mbox{ [7-(5,6)-5,3-(12,6)] , and  [7-(6,0)-0,3-(18,0)]. }\end{array} $$
As both lattices, $L$ and $\ ^{(3)}L^{\#} $ are extremal, this allows
us to apply Theorem \ref{main1} to conclude that the second and the 
third case is not possible. 
It remains to deal with the first case. 
Here the lattice $L_{\zeta} $ is in the genus 
$\II_{24}(3^{+12}7^{+4}) $ (see \cite{Juergens}) 
and $\sigma $ acts with minimal polynomial
$\Phi _7$ on $L_{\zeta }$, so we may consider $L_{\zeta } $ as 
a lattice of rank $r_{\zeta } = 4$ over $\Z[\zeta _7]$. 
Let $M_{\zeta }$ be a maximal even $\Z[\zeta _7]$-overlattice of $L_{\zeta }$. 
Then $M_{\zeta }$ is an even unimodular lattice of dimension 24
having an automorphism $\sigma $ with minimal polynomial $\Phi _7$. 
There are two such lattices, the Leech lattice $\Lambda _{24}$  and
the lattice with root system $A_6^4$. 
For both lattices there is up to conjugacy a unique automorphism 
$\sigma $ with the correct minimal polynomial. So $L_{\zeta }$ is
a sublattice of index $p_7 p_3$ of one of these two maximal 
$\Z[\zeta _7]$-lattices. We first construct all representatives of 
the isometry classes in the set of sublattices of index $p_7$, 
there are five such isometry classes of $\Z[\zeta _7]$-lattices 
of determinant $7^4$ (the mass of the genus of Hermitian lattices is 
$\frac{395}{345744}$). 
For all these five lattices we construct all  
$14,393,320$ sublattices $L_{\zeta }$ of
index $p_3$ of level $21$, without testing isometry. None of them 
has minimum $8$. 
\end{proof}

\begin{corollary}
The genus
$\II_{24}(3^{+12}7^{+4}) $ contains no lattice $L_{\zeta }$ of 
minimum $8$ admitting an automorphism $\sigma $ with minimal polynomial
$\Phi _7$.
\end{corollary}

The possible det-types of automorphisms of order $5$ of an extremal
$3$-modular lattice of dimension 36 are listed in \cite{Juergens} as 
$$ \begin{array}{l} 
	\mbox{ [5-(5,16)-5, 3-(8,10) ], 
[5-(6,12)-6, 3-(8,10) ], 
	[5-(6,12)-6, 3-(12,6) ], } \\
	\mbox{ [5-(7,8)-5, 3-(12,6) ], 
[5-(7,8)-7, 3-(12,6) ],  and 
[5-(8,4)-4, 3-(16,2) ].} \end{array}$$ 
By Theorem \ref{main1} only the third and the last possibility can occur.
A computation excluding the  possibilities for the third
case hence shows the following theorem. 

\begin{theorem} \label{aut5} 
Let $L$ be an extremal $3$-modular lattice of dimension $36$ and
$\sigma \in \Aut(L)$ be an automorphism of order $5$. 
Then the det-type of $(L,\sigma )$ is 
$[5-(8,4)-4, 3-(16,2) ]$.
\end{theorem}

\begin{proof} 
	As explained above we only need to exclude the 
	possibility that the det-type of $(L,\sigma )$ is
$[5-(6,12)-6, 3-(12,6) ] $.
So assume that we have such an automorphism $\sigma $ of order 5.

	Then the lattice 
	$L_{\zeta } $ is in the genus $\II _{24} (5^{-6} 3^{-12} )$ 
	(see \cite{Juergens}). To classify the relevant 
	$\Z [\zeta_5] $-lattices in this genus we first 
	classify the $\Z[\zeta _5]$-lattices in the 
	genus $\II _{24} (5^{-6} 3^{-4} )$. 
	The mass of this Hermitian genus is $577524389/405000000$ 
	and its class number is $222$. 
	Only $132$ of these $222$ lattices $L$ have the additional
	property that $\min (L^{\#,3}) \geq 8/3$. 
	For these $132$ lattices we compute the maximal $\zeta _5$-invariant
	sublattices of index $3^4$ that have minimum 8. 
	There are in total $3$ isometry classes of such lattices $M$,
	all satisfying $\min (M^{\# ,3}) \geq 8/3$. 
	These three lattices are candidates for $L_{\zeta }$. 

	For these lattices we computed all the 3-modular overlattices 
	of $L_{\zeta } \perp L_1$ where $L_1$ is one of the three
	extremal strongly $15$-modular lattices from Theorem \ref{dim12ext15}.
	None of the $3$-modular overlattices has minimum 8. 
\end{proof} 

The last theorem summarizes our present knowledge.
\begin{theorem}
	Let $\sigma $ be an 
automorphisms of prime order of an extremal $3$-modular lattice  $L$ 
of dimension $36$. 
Then the order of $\sigma $ is $2,3,5,$ or $7$. 
If $\sigma $ has order $7$ then it acts fixed point freely, i.e.
with minimal polynomial $\Phi _7$. 
If the order of $\sigma $ is $5$, then the det-type of 
$(L,\sigma )$ is  $[5-(8,4)-4, 3-(16,2) ]$.
\end{theorem}


\bibliography{dim36}

\begin{thebibliography}{10}

\bibitem{BachocNebe}
Christine Bachoc and Gabriele Nebe.
\newblock Extremal lattices of minimum {$8$} related to the {M}athieu group
  {$M_{22}$}.
\newblock {\em J. Reine Angew. Math.}, 494:155--171, 1998.
\newblock Dedicated to Martin Kneser on the occasion of his 70th birthday.

\bibitem{BachocVenkov}
Christine Bachoc and Boris Venkov.
\newblock Modular forms, lattices and spherical designs.
\newblock In {\em R\'{e}seaux euclidiens, designs sph\'{e}riques et formes
  modulaires}, volume~37 of {\em Monogr. Enseign. Math.}, pages 87--111.
  Enseignement Math., Geneva, 2001.

\bibitem{Magma}
W.~Bosma, J.~Cannon, and C.~Playoust.
\newblock The {M}agma algebra system. {I}. {T}he user language.
\newblock {\em J. Symbolic Comput.}, 24(3-4):235--265, 1997.

\bibitem{CohnKumar}
Henry Cohn and Abhinav Kumar.
\newblock Optimality and uniqueness of the {L}eech lattice among lattices.
\newblock {\em Ann. of Math. (2)}, 170(3):1003--1050, 2009.

\bibitem{ViaCohn}
Henry Cohn, Abhinav Kumar, Stephen~D. Miller, Danylo Radchenko, and Maryna
  Viazovska.
\newblock The sphere packing problem in dimension 24.
\newblock {\em Ann. of Math. (2)}, 185(3):1017--1033, 2017.

\bibitem{SPLAG}
J.~H. Conway and N.~J.~A. Sloane.
\newblock {\em Sphere packings, lattices and groups}, volume 290 of {\em
  Grundlehren der Mathematischen Wissenschaften [Fundamental Principles of
  Mathematical Sciences]}.
\newblock Springer-Verlag, New York, 1988.
\newblock With contributions by E. Bannai, J. Leech, S. P. Norton, A. M.
  Odlyzko, R. A. Parker, L. Queen and B. B. Venkov.

\bibitem{Ebeling}
Wolfgang Ebeling.
\newblock {\em Lattices and codes}.
\newblock Advanced Lectures in Mathematics. Springer Spektrum, Wiesbaden, third
  edition, 2013.
\newblock A course partially based on lectures by Friedrich Hirzebruch.

\bibitem{Feit}
Walter Feit.
\newblock Some lattices over {${\bf Q}(\surd -3)$}.
\newblock {\em J. Algebra}, 52(1):248--263, 1978.

\bibitem{Hales}
Thomas~C. Hales.
\newblock A proof of the {K}epler conjecture.
\newblock {\em Ann. of Math. (2)}, 162(3):1065--1185, 2005.

\bibitem{Juergens}
Michael J\"urgens.
\newblock Nicht-existenz und konstruktion extremaler gitter.
\newblock {\em Dissertation, TU Dortmund},
  http://dx.doi.org/10.17877/DE290R-7882, 2015.

\bibitem{King}
Oliver~D. King.
\newblock A mass formula for unimodular lattices with no roots.
\newblock {\em Math. Comp.}, 72(242):839--863, 2003.

\bibitem{KirschmerNebe}
Markus Kirschmer and Gabriele Nebe.
\newblock Binary hermitian lattices over number fields.
\newblock {\em Experimental Mathematics}, 2019.

\bibitem{Kneser}
M.~Kneser.
\newblock {\em Quadratische {F}ormen}.
\newblock Springer-Verlag, Berlin, 2002.
\newblock Revised and edited in collaboration with Rudolf Scharlau.

\bibitem{Martinet}
Jacques Martinet.
\newblock {\em Perfect lattices in {E}uclidean spaces}, volume 327 of {\em
  Grundlehren der Mathematischen Wissenschaften [Fundamental Principles of
  Mathematical Sciences]}.
\newblock Springer-Verlag, Berlin, 2003.

\bibitem{Morales}
Jorge~F. Morales.
\newblock Maximal {H}ermitian forms over {$\Bbb{Z}G$}.
\newblock {\em Comment. Math. Helv.}, 63(2):209--225, 1988.

\bibitem{Nebeautext}
G.~Nebe.
\newblock On automorphisms of extremal even unimodular lattices.
\newblock {\em Int. J. Number Theory}, 9(8):1933--1959, 2013.

\bibitem{Nebe48}
G.~Nebe.
\newblock A fourth extremal even unimodular lattice of dimension 48.
\newblock {\em Discrete Math.}, 331:133--136, 2014.

\bibitem{PreprintNebe}
G.~Nebe.
\newblock Automorphisms of prime order p of extremal p-modular lattices.
\newblock {\em in preparation}, 2019.

\bibitem{cycloquat}
Gabriele Nebe.
\newblock Some cyclo-quaternionic lattices.
\newblock {\em J. Algebra}, 199(2):472--498, 1998.

\bibitem{NebeJahresbericht}
Gabriele Nebe.
\newblock Gitter und {M}odulformen.
\newblock {\em Jahresber. Deutsch. Math.-Verein.}, 104(3):125--144, 2002.

\bibitem{Nebe72}
Gabriele Nebe.
\newblock An even unimodular 72-dimensional lattice of minimum 8.
\newblock {\em J. Reine Angew. Math.}, 673:237--247, 2012.

\bibitem{NebeVenkovsurv}
Gabriele Nebe.
\newblock Boris {V}enkov's theory of lattices and spherical designs.
\newblock In {\em Diophantine methods, lattices, and arithmetic theory of
  quadratic forms}, volume 587 of {\em Contemp. Math.}, pages 1--19. Amer.
  Math. Soc., Providence, RI, 2013.

\bibitem{LatDB}
Gabriele Nebe and Neil Sloane.
\newblock Catalogue of lattices.
\newblock {\em http://www.math.rwth-aachen.de/$\sim$Gabriele.Nebe/LATTICES/}.

\bibitem{NebeVenkov}
Gabriele Nebe and Boris~B. Venkov.
\newblock Nonexistence of extremal lattices in certain genera of modular
  lattices.
\newblock {\em J. Number Theory}, 60(2):310--317, 1996.

\bibitem{OMeara}
O.~T. O'Meara.
\newblock {\em Introduction to {Q}uadratic {F}orms}.
\newblock Springer, 1973.

\bibitem{Quebbemann}
H.-G. Quebbemann.
\newblock Atkin-{L}ehner eigenforms and strongly modular lattices.
\newblock {\em Enseign. Math. (2)}, 43(1-2):55--65, 1997.

\bibitem{Quebbemann2}
H.-G. Quebbemann.
\newblock Atkin-{L}ehner eigenforms and strongly modular lattices.
\newblock {\em Enseign. Math. (2)}, 43(1-2):55--65, 1997.

\bibitem{ScharlauHemkemeier}
R.~Scharlau and B.~Hemkemeier.
\newblock Classification of integral lattices with large class number.
\newblock {\em Math. Comp.}, 67(222):737--749, 1998.

\bibitem{SchaSchuPi}
R.~Scharlau and R.~Schulze-Pillot.
\newblock Extremal lattices.
\newblock In {\em Algorithmic algebra and number theory ({H}eidelberg, 1997)},
  pages 139--170. Springer, Berlin, 1999.

\bibitem{SiegelI}
C.~L. Siegel.
\newblock {\"U}ber die {A}nalytische {T}heorie der quadratischen {F}ormen.
\newblock {\em Annals of Mathematics}, 36(3):527--606, 1935.

\bibitem{SiegelIII}
C.~L. Siegel.
\newblock {\"U}ber die {A}nalytische {T}heorie der quadratischen {F}ormen
  {III}.
\newblock {\em Annals of Mathematics}, 38(1):212--291, 1937.

\bibitem{Siegel}
Carl~Ludwig Siegel.
\newblock Berechnung von {Z}etafunktionen an ganzzahligen {S}tellen.
\newblock {\em Nachr. Akad. Wiss. G\"{o}ttingen Math.-Phys. Kl. II},
  1969:87--102, 1969.

\bibitem{Viazovska}
Maryna~S. Viazovska.
\newblock The sphere packing problem in dimension 8.
\newblock {\em Ann. of Math. (2)}, 185(3):991--1015, 2017.

\bibitem{Voullie}
Fabian Voullie.
\newblock Involutionen extremaler modularer {G}itter.
\newblock {\em Masterthesis, RWTH Aachen University}, 2019.

\end{thebibliography}

\end{document}